\documentclass[12pt,a4paper,oneside]{amsart}
\usepackage{amsfonts, amsmath, amssymb, amsthm, amscd}

\usepackage{anysize}
\marginsize{2cm}{2cm}{1.5cm}{1.5cm}

\usepackage[english]{babel}
\usepackage{graphicx}
\usepackage{url}
\usepackage{caption}
\usepackage{float}
\usepackage{tikz-cd}
\usepackage{wrapfig}
\usepackage{enumitem} 

\usepackage{hyperref}
\hypersetup{
   colorlinks   = true, %Colours links instead of ugly boxes
   urlcolor     = blue, %Colour for external hyperlinks
   linkcolor    = blue, %Colour of internal links
   citecolor   = red %Colour of citations
}

\def\Z{{\mathbb Z}}
\def\R{{\mathbb R}}

\newtheorem{thm}{Theorem}[section]
\newtheorem{lem}[thm]{Lemma}

\newtheorem{cor}[thm]{Corollary}
\newtheorem{que}[thm]{Question}

\newtheorem*{thm*}{Theorem}

\theoremstyle{remark}

\newtheorem{rem}[thm]{Remark}

\theoremstyle{definition}

\numberwithin{equation}{section}

\title{Vanishing of all equivariant obstructions and the mapping degree}

\author{Sergey~Avvakumov{$^\spadesuit$}}

\author{Sergey~Kudrya{$^\clubsuit$}}

\thanks{{$^\spadesuit$} Has received funding from the Austrian Science Fund (FWF), Project P31312-N35, and the European Research Council under the European Union’s Seventh Framework Programme ERC Grant agreement ERC StG 716424 - CASe.}

\thanks{{$^\clubsuit$} Supported by the Austrian Academic Exchange Service (OeAD), ICM-2019-13577}

\address{Sergey~Avvakumov, Department of Mathematical Sciences, University of Copenhagen, Copenhagen, Denmark}
\email{savvakumov@gmail.com}

\address{Sergey~Kudrya, Moscow Institute of Physics and Technology, Institutskiy per. 9, Dolgoprudny, Russia 141700}
\email{sergeykudrya@bk.ru}


\begin{document}

\begin{abstract}
Suppose that $n$ is not a prime power and not twice a prime power. We prove that for any Hausdorff compactum $X$ with a free action of the symmetric group $\mathfrak S_n$ there exists an $\mathfrak S_n$-equivariant map $X \to \R^n$ whose image avoids the diagonal $\{(x,x\dots,x)\in \R^n|x\in \R\}$.

Previously, the special cases of this statement for certain $X$ were usually proved using the equivartiant obstruction theory. Such calculations are difficult and may become infeasible past the first (primary) obstruction. We take a different approach which allows us to prove the vanishing of all obstructions simultaneously. The essential step in the proof is classifying the possible degrees of $\mathfrak S_n$-equivariant maps from the boundary $\partial\Delta^{n-1}$ of $(n-1)$-simplex to itself.

Existence of equivariant  maps between spaces is important for many questions arising from discrete mathematics and geometry, such as Kneser's conjecture, the Square Peg conjecture, the Splitting Necklace problem, and the Topological Tverberg conjecture, etc. We demonstrate the utility of our result applying it to one such question, a specific instance of envy-free division problem.

\end{abstract}

\maketitle

\section{Main result}

%%%%%%%%%%%%%%%%%%%%%%%%%%%%%%%%%%%%%%%%%%%%%%%%%%%%%%%%%%%%%%%%%%%%%%%%%%%%%%%%%%%%%%%%%%%%%%%%%%%%%%%%%%%%%%%%%
%%%%%%%%%%%%%%%%%%%%%%%%%%%%%%%%%%%%%%%%%%%%%%%%%%%%%%%%%%%%%%%%%%%%%%%%%%%%%%%%%%%%%%%%%%%%%%%%%%%%%%%%%%%%%%%%%

The well-known framework of configuration spaces, test maps, and equivariant obstructions has been an extremely fruitful method for dealing  with questions arising from discrete mathematics  and geometry, and has been used with great success to solve a variety of problems, including Kneser's conjecture \cite{lovasz1978kneser}, the Square Peg conjecture for smooth curves \cite{shnirel1944certain}, the Splitting Necklace problem \cite{alon1987splitting}, and the Topological Tverberg conjecture \cite{barany1981topological, ozaydin1987equivariant, volovikov1996topological}, etc. 

Let us briefly describe the framework: One starts with defining a suitable configuration space of potential solutions to the problem. Then an appropriate test map from the configuration space to the test space is defined. Informally, the test map measures how far the given potential solution is from the target, a subspace of the test space. A point in the configuration space is a valid solution to the problem if and only if its image under the test map intersects (``hits'') the target. Typically, a certain symmetry group defined by the problem acts both on configuration and test spaces, and the test map is equivariant with respect to this action. So, one can now restate the problem in topological terms: Is it true that any equivariant map from the configuration space to the test space hits the target? If the answer is ``yes'', then the original problem always has a solution. 

The ``straightforward'' way to answer this topological question is to compute a series of equivariant topological obstructions, the larger the difference between the dimensions of the configuration and test spaces are, the longer the series. In practice, computing even the second obstruction may be very challenging if not impossible.

Often, the symmetry group of the problem is the symmetric group $\mathfrak S_n$; and the corresponding test space is $\R^n$ (or $(\R^n)^{\oplus d}$) with the diagonal $D_n:=\{(x,x\dots,x)\in \R^n|x\in \R\}$ (or $(D_n)^{\oplus d}$) as the target. For example, this is the case in the Splitting Necklace problem and the Topological Tverberg conjecture mentioned above, and fair or envy-free  division problems \cite{akopyan2018convex, avvakumov2019envy, blagojevic2014convex, gale1984equilibrium, karasev2014convex, meunier2019envy, segal2018fairly}. 

For this popular combination of the symmetry group and the test space, our main result allows to bypass the difficult calculations of the equivariant obstruction entirely:

\begin{thm}
\label{thm:cor_main}
Suppose that $n$ is not a prime power and not twice a prime power.
Then for any Hausdorff compactum $X$ with a free action
of $\mathfrak S_n$ there exists an equivariant map $X \to \R^n\setminus D_n$.
\end{thm}
The space $X$ in the statement is a substitute for the configuration space. The restrictions on $X$ are not significant, in practice a configuration space can usually be equivariantly contracted to a compact polyhedron.

Theorem \ref{thm:cor_main} fails when $n$ is a prime power, see Theorem \ref{thm:dold} below. As far as know it is an open question whether Theorem \ref{thm:cor_main} also fails when $n$ is twice a prime power, but we conjecture that it does.

The following similar theorem was recently proved in \cite{avvakumov2019stronger} and applied to get new and stronger counterexamples to the Topological Tverberg conjecture. To state the theorem we will first need the notation:
$$D_{d,n}:=\{(x,x\dots,x)\in \R^{d\times n}|x\in \R^d\}.$$

\begin{thm}
\label{thm:tvr}
Suppose that $n$ is not a prime power.
Then for any Hausdorff compactum $X$ with a free action
of $\mathfrak S_n$ there exists an equivariant map $X \to \R^{2\times n}\setminus D_{2,n}$.
\end{thm}

Note that $\R^n\setminus D_n=\R^{1\times n}\setminus D_{1,n}$ and that there is a natural equivariant inclusion 
\begin{equation}
\label{eq:inc}
\R^{d\times n}\setminus D_{d,n}\subset \R^{d'\times n}\setminus D_{d',n} \text{\quad when \quad } d < d'.
\end{equation}

So, Theorem \ref{thm:tvr} follows from Theorem \ref{thm:cor_main} when $n$ is not twice a prime power. On the other hand, when $n$ is twice a prime power Theorem \ref{thm:tvr} holds, while Theorem \ref{thm:cor_main} remains an open question, see Question \ref{que:2pk}.

From the point of view of the Topological Tverberg conjecture, our Theorem \ref{thm:cor_main} does not add anything new compared to Theorem \ref{thm:tvr}. The reason is that to build counterexamples to the conjecture in $\R^d$ one needs an equivariant map to $\R^{d\times n}\setminus D_{d,n}$. In light of the inclusion \eqref{eq:inc}, Theorem \ref{thm:tvr} already provides the required equivariant maps for all $d\geq 2$.

Theorem \ref{thm:cor_main} follows as the combination of two other statements, which are also useful by themselves:

\begin{lem}
\label{lem:Vol}
Let $G$ be a finite group and $S$ be a sphere with an action of $G$. If there exists an
equivariant map $f: S \to S$ of zero degree then any Hausdorff compactum $X$ with a free action
of $G$ has an equivariant map $X \to S$.
\end{lem}

See a proof of Lemma \ref{lem:Vol} and some historical remarks in \cite{avvakumov2019envy}.

To get Theorem \ref{thm:cor_main}, we would like to apply the lemma with $\mathfrak S_n$ and $\partial \Delta^{n-1}$ as $G$ and $S$, respectively. Then notice that $\partial \Delta^{n-1}$ is $\mathfrak S_n$-equivariantly homotopically equivalent to $\R^n\setminus D_n$. It remains to prove that there exists a $\mathfrak S_n$-equivariant map $\partial \Delta^{n-1}\to \partial \Delta^{n-1}$ of zero degree which is required for the application of the lemma. Instead of proving only that, we give an (incomplete) classification of all possible degrees of  $\mathfrak S_n$-equivariant maps $\partial \Delta^{n-1}\to \partial \Delta^{n-1}$:

\begin{thm}
\label{thm:main}
For $n>1$ consider the boundary $\partial \Delta^{n-1}$ of a standard simplex with the natural action of the symmetric group $\mathfrak S_n$ permuting the vertices.
Let $d$ be the degree of a $\mathfrak S_n$-equivariant map $\partial \Delta^{n-1}\to \partial \Delta^{n-1}$. Then:

\begin{enumerate}[label={(\alph*)}]
\item if $n=p^k$ for some prime $p\neq 2$ then $d$ can attain any value $d\equiv 1\pmod{p}$ and only such values,
\item if $n=2p^k$ for some prime $p$ then $d$ can only attain values $d\equiv \pm 1\pmod{p}$,
\item if $n$ is odd and $n \neq p^k$ for all primes $p$ then $d$ can attain any value,
\item if $n$ is even and $n \neq 2p^k$ for all primes $p$ then $d$ can attain $0$.
\end{enumerate}
\end{thm}

Only parts (c) and (d) of Theorem \ref{thm:main} are required to prove Theorem \ref{thm:cor_main}. The ``only'' part of Theorem \ref{thm:main}(a) was probably known before, and Theorem \ref{thm:main}(c) was first proved in \cite{avvakumov2019envy}. For the reader's convenience we prove Theorem \ref{thm:main} in full which takes up most of the paper.

The (almost) converse of Theorem \ref{thm:cor_main} holds for $n$ a prime power:

\begin{thm}
\label{thm:dold}
Suppose that $n$ is a prime power. Let $X$ be an $(n-2)$-connected topological space with an action of $\mathfrak S_n$. Then there is no $\mathfrak S_n$-equivariant map $X \to \R^n\setminus D_n$. 
\end{thm}

Theorem \ref{thm:dold} along with its proof was communicated to us by Roman~Karasev. The theorem follows from \cite[Propositions 2.1, 2.5, 3.2]{volovikov2000index} but was probably known much earlier. For the reader's convenience we sketch the main steps of the proof communicated by Karasev at the end of the paper.

A natural question is if an analogue of Theorem \ref{thm:dold} holds for $n$ twice a prime power:
\begin{que}
\label{que:2pk}
Suppose that $n = 2p^k$ for some prime $p\neq 2$.

Is there $d$ such that for every $d$-connected Hausdorff compactum $X$ with a free action
of $\mathfrak S_n$ there is no equivariant map $X \to \R^n\setminus D_n$? If so, what is the minimal such $d$? 
\end{que}

\begin{rem}
For $X$ satisfying the conditions of Lemma \ref{lem:Vol} there is even a $\mathfrak S_n$-equivariant map $X*(\R^n\setminus D_n) \to \R^n\setminus D_n$ (follows from the proof of the lemma in \cite{avvakumov2019envy}). Note that the action on $X*(\R^n\setminus D_n)$ is not free. It is interesting to know when there exists a $\mathfrak S_n$-equivariant map $Y\to \R^n\setminus D_n$ for a general $Y$ with a non-free action of $\mathfrak S_n$. 
\end{rem}

\subsection*{Acknowledgments}
We thank Roman~Karasev and an anonymous referee for useful remarks and corrections to the text.

\section{Application to envy-free division}
We consider the following \emph{envy-free division problem} similar to the convex fair partition problem in \cite{akopyan2018convex, blagojevic2014convex, karasev2014convex}. Let $K\subset \R^d$ be a convex body. Consider $n$ players which want to divide $K$ among themselves into convex pieces of equal volume and equal subjective value: for each division of $K$ into convex pieces of equal volume each player has one or several pieces they like most. The preferences of the players are ``continuous'', i.e., the set of divisions of $K$ into convex pieces of equal volume where $i$th player likes the $j$th piece is closed. A division of $K$ into $n$ pieces \emph{solves} the problem if we can match players with pieces they like.

\begin{thm}
\label{thm:app}
Suppose that the number of players $n$ is not a prime power and not twice a prime power.

Then for any $d$ and convex $K\subset \R^d$ there is an instance of the envy-free division problem with no solution. Moreover, in this instance all players have the same preferences.
\end{thm}
\begin{proof}
Consider the space $X$ of all divisions of $K$ into $n$ labeled convex pieces of equal volume. The symmetric group $\mathfrak S_n$ acts on $X$ by permuting the labels. 
Denote by $F(n, d)$ the space of ordered $n$-tuples of pairwise distinct points in $\R^d$. The group $\mathfrak S_n$ acts on $F(n, d)$ by permuting the points.
There is a $\mathfrak S_n$-equivariant map $X\to F(n, d)$ which maps each piece to its center of mass.

The space $F(n, d)$ retracts $\mathfrak S_n$-equivariantly to a compact polyhedron, see \cite{blagojevic2014convex}. So, by Theorem \ref{thm:cor_main}, there is a $\mathfrak S_n$-equivariant map $F(n, d)\to \R^n\setminus D_n$. Composing it with the map $X\to F(n, d)$ we get that there is a $\mathfrak S_n$-equivariant $f:X\to \R^n\setminus D_n$. (Note, that here Theorem \ref{thm:tvr})

Now assume that all players have the same preferences. In particular, that each player likes those pieces in a given division $x\in X$ which maximize the corresponding coordinate of $f(x)$. Player's preference does not change if we renumber the pieces in $x$ because $f$ is $\mathfrak S_n$-equivariant. A division $x\in X$ solves this envy-free division problem if and only if $f(x)\in D_n$, which is impossible.
\end{proof}

Note that the use of Theorem \ref{thm:cor_main} in the proof cannot be substituted by Theorem \ref{thm:tvr}. The reason is that Theorem \ref{thm:tvr} states only the existence of a map to $\R^{2\times n}\setminus D_{2,n}$ while we need a map to $\R^{1\times n}\setminus D_{1,n}=\R^n\setminus D_n$.

Combining the approach from \cite{gale1984equilibrium} with the results of \cite{karasev2014convex} or \cite{blagojevic2014convex}, we can prove that

\begin{thm}
If $d\geq 2$ and $n$ is a power of a prime, then the envy-free division problem always has a solution.
\end{thm}
\begin{proof}
As in the proof of Theorem \ref{thm:app}, denote by $F(n, d)$ the space of $n$-tuples of pairwise distinct points in $\R^d$. To each element of $x\in F(n, d)$ corresponds a unique division of $K$ into convex pieces of equal volume. This division is the intersection of $K$ with a generalized Voronoi diagram with centers being the points in the $n$-tuple $x$, see \cite{karasev2014convex} or \cite{blagojevic2014convex}. So, we can identify $F(n, d)$ with some subset of divisions of $K$ into convex pieces of equal volume.

Using the approach of \cite{gale1984equilibrium} (also described in \cite{avvakumov2019envy}), we can ``convert'' player's preferences to a $\mathfrak S_n$-equivariant function $f:F(n, d)\to \R^n$ such that $x\in F(n, d)$ solves the problem if $f(x)\in D_n$. And for $d\geq 2$ and $n$ a prime power any $\mathfrak S_n$-equivariant map $f:F(n, d)\to \R^n$ hits the diagonal $D_n$, see \cite[Theorem 1.10]{karasev2014convex} or \cite[Theorem 1.2]{blagojevic2014convex}, which guarantees the existence of the solution.
\end{proof}

We do not know if our envy-free division problem has a solution for $n$ twice a prime power.

\section{Proof of Theorem \ref{thm:main}}
%%%%%%%%%%%%%%%%%%%%%%%%%%%%%%%%%%%%%%%%%%%%%%%%%%%%%%%%%%%%%%%%%%%%%%%%%%%%%%%%%%%%%%%%%%%%%%%%%%%%%%%%%%%%%%%%%
%%%%%%%%%%%%%%%%%%%%%%%%%%%%%%%%%%%%%%%%%%%%%%%%%%%%%%%%%%%%%%%%%%%%%%%%%%%%%%%%%%%%%%%%%%%%%%%%%%%%%%%%%%%%%%%%%

Denote by $\Sigma_n$ the boundary $\partial\Delta^{n-1}$. We agree that $\Sigma_1=\emptyset$ and ${\rm dim}\Sigma_1=-1$.

\begin{lem}
\label{lem:mod}
Let $G\subseteq \mathfrak S_n$ be a subgroup.
There exists a $G$-equivariant map $\Sigma_n\to\Sigma_n$ of degree $d$ if and only if
there exist points $x_1,\dots,x_k\in \Sigma_n$, subgroups $G_1,\dots,G_k\subseteq G$, and integers $d_1,\dots,d_k$, such that
\[
d = 1 - \sum_{i=1}^k d_i\frac{|G|}{|G_i|},
\]
and for each $i=1,\dots,k$
\begin{itemize}
\item[(1)] the subgroup $G_i\subseteq G$ is the stabilizer of $x_i\in \Sigma_n$,
\item[(2)] the $G$-orbits of all $x_i$ are pairwise disjoint,
\item[(3)] there is a $G_i$-equivariant map $\Sigma_n\to\Sigma_n$ of degree $d_i$ which is an identity in a neighborhood of $x_i$.
\end{itemize}
\end{lem}

To prove the ``only if'' part of Lemma \ref{lem:mod} we need the following technical lemma. Informally, this lemma says that given a generic equivariant map on a subcomplex we can extend it, under some dimensional conditions, to the whole complex generically and equivariantly. See also the absolute version of the lemma (the case $Q=\emptyset$) in \cite[Lemma 3.4]{avvakumov2019envy} \footnote{Lemma 3.4 in the published version of \cite{avvakumov2019envy} has an incorrect proof and is probably false. The mistake is corrected in the arXiv versions 7 and later.}.

\begin{lem}
\label{lem:finite-to-one}
Assume that $G$ is a finite group acting on a simplicial complex $P$ and acting linearly on a vector space $V$. Let $Q\subset P$ be a $G$-equivariant subcomplex. Assume that for every subgroup $H\subseteq G$ the inequality $\dim P^H \le\dim V^H$ holds for the subspaces of $H$-fixed points in $P$ and $V$, respectively.

Let $y\in V$ be a point and let $f:Q\to V$ be a $G$-equivariant continuous map linear on every simplex of $Q$ and such that $f^{-1}(y)$ is finite. Then there exists a $G$-equivariant PL map $F:P\to V$ such that $F^{-1}(y)$ is finite and $F|_Q=f$.
\end{lem}

\begin{proof}
Consider the barycentric subdivision of $P$. The vertices of the barycentric subdivision are marked by the dimension of the faces they originate from and those marks are preserved by $G$. Hence the action of $G$ has the following property:

\begin{itemize}
\item [(*)] For any $g\in G$ and any face $\sigma$ of the barycentric subdivision we have $g\sigma = \sigma$ if and only if $g$ is the identity map on $\sigma$.
\end{itemize}

We now assume that the triangulation of $P$ has this property. 

Let $P'$ be the barycentric subdivision of $P$ (that would be the second subdivision if the original triangulation didn't satisfy (*)).
We shall define $F$ on the vertices of $P'$ and extend $F$ to the simplices of $P'$ by linearity.

For every vertex of $P'$ which lies in $Q$ define $F$ to be equal to $f$.

For a vertex of $P'$ its \emph{dimension} is the dimension of the corresponding simplex of $P$. We proceed to define $F$ on the $G$-orbits of the remaining vertices in order of non-decreasing dimension proving that after every step  $(F|_{\sigma'})^{-1}(y)$ is finite for every simplex $\sigma'$ of $P'$ on which $F$ is already defined.

It remains to describe the individual step. Let $x$ be one of the lowest dimensional vertices of $P'$ on which $F$ is not defined yet. Let $\sigma\subset P$ be the simplex of $P$ of which $x$ is the center. Map $F$ is already defined on $\partial\sigma$ and by our construction we have that $(F|_{\partial\sigma})^{-1}(y)$ is finite. Let $H\subseteq G$ be the stabilizer of $x$. The value $F(x)$ must be chosen in $V^H$ and $F(x)\in V^H$ is the only constraint needed to extend $F$ to the orbit $Gx$ equivariantly.

By (*) we have that every point of $\sigma$ is fixed by $H$, hence $\sigma\subset P^H$. By the assumption of the lemma, $\dim\sigma \le \dim V^{H}$.

For every simplex $\tau \subset \partial\sigma\subset P'$ 

\begin{itemize}
\item [(i)] if $\dim F(\tau) = \dim V^H-1$, and hence $\dim F(\tau)=\dim \tau$, choose $F(x)$ outside of the affine hull of $F(\tau)$;
\item [(ii)] if $\dim F(\tau) < \dim V^H-1$ choose $F(x)$ outside of the affine hull of $F(\tau)\cup y$.
\end{itemize}

In both (i) and (ii) the ``forbidden'' affine hulls have dimension at most $\dim V^H-1$ and hence there is a choice of $F(x)\in V^H$ which satisfies these rules for all $\tau \subset \partial\sigma$.

In case (i) we have that the vertices of $\tau*x$ are mapped to affinely independent points and so $F(\tau*x)$ covers $y$ at most once.
In case (ii) we have that $|(F|_{\tau*x})^{-1}(y)|=|(F|_{\tau})^{-1}(y)|$, i.e., the $F$-image of $(\tau*x)\setminus \tau$ does not intersect $y$.
To summarize, we get that $(F|_{\sigma})^{-1}(y)$ is finite.
\end{proof}

Denote by $W_n$ the affine span of $\Delta^{n-1}$. We consider $W_n$ as a linear space with $0$ at the center of $\Delta^{n-1}$. Sometimes we identify $\Sigma_n\subset W_n$ with the unit sphere in $W_n$ by a $\mathfrak S_n$-equivariant homeomorphism.

\begin{proof}[\it Proof of the ``only if'' part of Lemma \ref{lem:mod}]
Consider any $G$-equivariant map $\Sigma_n \to \Sigma_n$ and compose it with the inclusion $\Sigma_n\subset W_n$ to obtain a $G$-equivariant map
\[
f_1 : \Sigma_n \to W_n.
\]
Let $f_0 : \Sigma_n\to W_n$ be the standard $\mathfrak S_n$-equivariant inclusion. 
By Lemma \ref{lem:finite-to-one}, we can connect $f_0$ and $f_1$ by a $G$-equivariant homotopy 
\[
h : \Sigma_n\times [0,1] \to W_n,
\]
such that $h^{-1}(0)$ is finite. (Let us check that Lemma \ref{lem:finite-to-one} can be applied. In the notation of the lemma we have $V=W_n$, $P=\Sigma_n\times [0,1]$, $Q=\Sigma_n\times \{0,1\}$, $y=0$, $f=f_0\sqcup f_1$, $F=h$. The space $\Sigma_n\times [0,1]$ can be identified with the closure of $2\Delta^{n-1}\setminus\Delta^{n-1}\subset W_n$, where $2\Delta^{n-1}$ is the unit simplex scaled by the factor of $2$. Under this identification we have $(\Sigma_n\times [0,1])^H\subset W_n^H$ for every subgroup $H\subseteq G$, and hence $\dim(\Sigma_n\times [0,1])^H\leq \dim W_n^H$. Also, $(f_0\sqcup f_1)^{-1}(0)=\emptyset$ and so is finite. So, the hypothesis of the lemma is satisfied.) 

Note that the difference in the degrees of $f_0$ and $f_1$ as maps of $\Sigma_n$ to itself equals the degree of $h$ over the center $0\in \Delta^{n-1}$. This follows from the fact that the degree of a map between closed connected oriented manifolds with boundary $h : M\to N$ satisfying $h(\partial M)\subset \partial N$ is well defined and equals the degree of the restriction $h|_{\partial M} : \partial M\to \partial N$ if $\partial N$ is connected. Here $M=\partial \Delta^{n-1}\times [0,1]$ and $N=\Delta^{n-1}$.

A \emph{local degree} of $h$ at a point of $h^{-1}(0)$ is the degree of the restriction of $h$ to a small neighborhood of the point containing no other points of $h^{-1}(0)$.
Note, that local degree can take integer values different from $\pm 1$ if $0$ is not a regular value of the restriction.
Because $h^{-1}(0)$ is finite, the degree of $h$ can be counted geometrically as the sum of local degrees at the points of $h^{-1}(0)$.

Let $(x_i, t_1)\in h^{-1}(0)$ be a point and let $(x_i, t_2), \dots, (x_i, t_\ell)\in h^{-1}(0)$ be all the other points of $h^{-1}(0)$ whose first coordinate is $x_i$. For any $j$ and $\sigma\in G\subseteq\mathfrak S_n$ the degree at the point $(\sigma x_i,t_j)$ is the same as at $(x_i,t_j)$, because $\sigma$ acts on the orientation of both the domain and the range by the permutation sign. So, the total degree corresponding to $x_i$ is $-d_i\frac{|G|}{|G_i|}$, where $-d_i$ is the sum over $j$ of the degrees at $(x_i,t_j)$ and $\frac{|G|}{|G_i|}$ is the size of the orbit of $x_i$. It remains to prove that $d_i$ satisfies (3).

Let $U\subset \Sigma_n$ be a neighborhood of $x_i$ such that $G_iU=U$. We take $U$ sufficiently small so that $U\times[0,1]$ contains no points of $h^{-1}(0)$ except for $(x_i, t_1), \dots, (x_i, t_\ell)$.
Clearly, $d_i$ equals the degree of the map 
\[
\phi:\partial(U\times[0,1]) \xrightarrow{h} W_n\setminus 0\xrightarrow{{\rm pr}} \Sigma_n,
\]
where ${\rm pr}:W_n\setminus 0\to \Sigma_n$ is the standard radial projection. The map $\phi$ is $G_i$-equivariant as a composition of two $G_i$-equivariant maps. The restriction of $\phi$ to $U\times \{0\}$ is the identity. 

There exists a $G_i$-equivariant homeomorphism $\psi:\partial(U\times[0,1])\to \Sigma_n$ which is the identity on $U\times \{0\}$. For example, one can construct $\psi$ as follows. Let $\psi':\partial(U\times[0,1])\to W_n$ be the map which is the identity on $U\times \{0\}$, maps every $y\times \{1\}\in U \times \{1\}$ to $y-2x_i$ (here we consider $y$ and $x_i$ as vectors in $W_n$), and is linear in $t\in[0,1]$ on $\partial U\times [0,1]$. Let $\psi$ be the composition of $\psi'$ with the projection ${\rm pr}:W_n\setminus 0\to \Sigma_n$.

So, $\phi\circ \psi^{-1}:\Sigma_n\to\Sigma_n$ is a $G_i$-equivariant map of degree $d_i$ and the identity on $U\ni x_i$, hence $d_i$ satisfies (3).
\end{proof}

\begin{proof}[\it Proof of the ``if'' part of Lemma \ref{lem:mod}]
For each $i=1,\dots,k$, let $g_i:\Sigma_n\to \Sigma_n$ be a $G_i$-equivariant map of degree $d_i$ which is the identity in a small neighborhood $U_i$ of $x_i$. 

There is a smaller neighborhood $x_i\in V_i\subset U_i$ such that $G_iV_i=V_i$ and there is $G_i$-equivariant homeomorphism $\phi_i:V_i\to\Sigma_n\setminus V_i$ which is the identity on $\partial V_i$. One can construct $V_i$ and $\phi_i$ as follows. Until the end of the paragraph identify $\Sigma_n$ with the unit sphere $S\subset W_n$ by a $\mathfrak S_n$-equivariant homeomorphism. Choose $V_i\subset U_i$ to be a small spherical neighborhood of $x_i$. Clearly, $G_iV_i=V_i$ (here we extend the action of $\mathfrak S_n\supset G$ to $W_n$ in the natural way). Denote by $C$ the point outside $S$ and lying on the line connecting $0$ with $x_i$ and such that any line connecting $C$ to any point in $\partial V$ is tangent to $S$. Define $\phi_i:V\to S\setminus V$ to be the radial projection with center $C$.

Because the $G$-orbits of all $x_i$ are pairwise disjoint and because we can choose $V_i$ to be arbitrary small, we may assume that
\begin{itemize}
\item the $G$-orbits of all $V_i$ are also pairwise disjoint,
\item $V_i\cap Gx_i=\{x_i\}$.
\end{itemize}

Define a map $f : \Sigma_n\to \Sigma_n$ as follows:
\begin{itemize}
\item $f$ equals to $g_i\circ \sigma \circ  \phi_i \circ \sigma^{-1}$ on $\sigma V_i$ for all $\sigma\in G$ and $1\leq i\leq k$,
\item $f$ is the identity map elsewhere.
\end{itemize}

It is easy to check that $f$ is continuous, well defined (independent of $\sigma$), and $G$-equivariant.

Clearly, ${\rm deg}f=1 - \sum_{i=1}^k {\rm deg}g_i\frac{|G|}{|G_i|}=1-\sum_{i=1}^k d_i\frac{|G|}{|G_i|}$.
\end{proof}

We are now ready to prove the ``only'' parts of Theorem \ref{thm:main}.
\begin{proof}[\it Proof of the ``only'' part of Theorem \ref{thm:main}(a) and Theorem \ref{thm:main}(b)]
Suppose that $n=p^k$ for some prime $p$. Consider any point of $\Sigma_n$ and split its barycentric coordinates into blocks of equal coordinates. Suppose the sizes of the blocks are $\alpha_1,\dots,\alpha_\ell$. Then the orbit of the point under $\mathfrak S_n$ has size 
\[
\frac{n!}{\alpha_1!\cdot\dots\cdot\alpha_\ell!} = \binom{n}{\alpha_1,\dots,\alpha_\ell}.
\]
The multinomial coefficient above is a product of binomial coefficients
\[
\binom{n}{\alpha_1,\dots,\alpha_\ell} = \binom{n}{\alpha_1}\cdot\binom{n-\alpha_1}{\alpha_2}\cdot,\dots,\cdot\binom{n-\alpha_1-\dots-\alpha_{\ell-1}}{\alpha_\ell}.
\]
Hence, it is divisible by $p$, as $\alpha_1<n$ and thus the first factor is divisible by $p$ by Lucas's theorem, \cite{lucas1878theorie}. So, the size of every orbit is divisible by $p$. Hence, by the ``only if'' part of Lemma \ref{lem:mod}, the degree of any $\mathfrak S_n$-equivariant map $\Sigma_n\to \Sigma_n$ is $1$ modulo $p$. This finishes the proof of the ``only'' part of Theorem \ref{thm:main}(a).

Suppose now that $n=2p^k$ for some prime $p\neq 2$. Then there is only one $\mathfrak S_n$ orbit whose size is not divisible by $p$. More precisely, it is the orbit of the center $x$ of the subsimplex of $\Sigma_n$ on the first $p^k$ vertices. Indeed, considering the product of binomial coefficients above and applying Lucas's theorem we see that the first factor is not divisible by $p$ only if $\alpha_1=p^k$ (note, that $\alpha_1=2p^k$ is impossible). Then the second factor is not divisible by $p$ only if $\alpha_2=\alpha_1=p^k$. The stabilizer of $x$ is $\mathfrak S_{p^k}\times\mathfrak S_{p^k}=:G$. The orbit of $x$ has size $\frac{|\mathfrak S_n|}{|G|}=\binom{2p^k}{p^k}$ which by Lucas's theorem equals $2$ modulo $p$.

So, by the ``only if'' part of Lemma \ref{lem:mod}, the degree of any $\mathfrak S_n$-equivariant map $\Sigma_n\to \Sigma_n$ is equal modulo $p$ either to $1$ (if none of the $G_i$ in the statement of the lemma are equal to $G$) or to $1-{\rm deg}(f)\frac{|\mathfrak S_n|}{|G|}\equiv 1-2\cdot{\rm deg}(f)\pmod{p}$, where $f:\Sigma_n\to \Sigma_n$ is some $G$-equivariant map which is the identity in a neighborhood of $x$.
It remains to prove that ${\rm deg}(f)$ is either $0$ or $1$ modulo $p$.

Let $x'$ be the center of the subsimplex of $\Sigma_n$ on the last $p^k$ vertices. Points $x$ and $x'$ are opposite to each other and are the only points of $\Sigma_n$ fixed by $G$. The size of the $G$-orbit of any other point $y\in\Sigma_n$ is divisible by $p$. Indeed, split the first $p^k$ barycentric coordinates of $y$ into blocks of equal coordinates. Suppose the sizes of the blocks are $\alpha_1,\dots,\alpha_\ell$. As before, the size of the $G$-orbit of $y$ is divisible by $\binom{p^k}{\alpha_1}$. By Lucas's theorem, the only way $\binom{p^k}{\alpha_1}$ is not divisible by $p$ is if $\alpha_1=p^k$, meaning that the first $p^k$ coordinates of $y$ are the same. Likewise, the last $p^k$ coordinates of $y$ are also the same. So, $y=x$ or $y=x'$.

Let $U$ be a small $G$-equivariant neighborhood of $x$ on which $f$ is the identity. 
Let us define a $G$-equivariant homotopy $h:\Sigma_n\times[0,1]\to W_n$ in the following way:

\begin{itemize}
\item $h|_{\Sigma_n\times \{0\}}=f$,
\item $h(\Sigma_n\times \{1\})=f(x')$,
\item $h(x'\times [0,1])=f(x')$,
\item $h|_{U\times[0,1]}$ is linear in $t\in [0,1]$, i.e., $h(y,t) = (1-t) h(y, 0) + t h(y, 1)$ for all $y\in U$.
\end{itemize}

By Lemma \ref{lem:finite-to-one}, we can $G$-equivariantly extend $h$ to the rest of $\Sigma_n\times[0,1]$ so that $h^{-1}(0)$ is finite and the degree can be counted geometrically as the sum of local degrees at the points of $h^{-1}(0)$ (see the definition of the local degree and a detailed explanation of why Lemma \ref{lem:finite-to-one} is applicable in the proof of the ``only if'' part of Lemma \ref{lem:mod} above).

The degree of the constant map $h|_{\Sigma_n\times \{1\}}$, considered as a map to $\Sigma_n$, is zero. So, the degree of $f$ is equal to the degree of $h$ over $0\in W_n$.

We know that $f(x)=x$ and from the $G$-equivariancy of $f:\Sigma_n\to \Sigma_n$, we have that $f(x')=x$ or $f(x')=x'$. 
Because $h$ is the identity on $U\times 0$ and linear in $t\in[0,1]$ on $U\times [0,1]$, we have that $(h|_{U\times[0,1]})^{-1}(0)$ is either empty, if $f(x')=f(x)=x$; or contains the single point of the form $(x,\frac{1}{2})$, if $f(x')=x'$. In the second case the local degree of $h$ at $(x,\frac{1}{2})$ is $1$.

Because $h(x', t)=f(x')\neq 0$ we get that $(x', t)$ is not in $h^{-1}(0)$ for every $t\in[0,1]$.

For any other $y\in \Sigma_n$, $y\neq x,x'$ the size of the $G$-orbit of $(y, t), t\in [0,1]$ is divisible by $p$.
So, the degree of $h$ over $0\in W_n$, and hence the degree of $f$, is either $1$ or $0$ modulo $p$, depending on whether $(x,\frac{1}{2})$ is in $h^{-1}(0)$ or not.  This finishes the proof of Theorem \ref{thm:main}(b).
\end{proof}

To prove the rest of Theorem \ref{thm:main} we investigate which degrees can be attained by maps $\Sigma_n\to\Sigma_n$ satisfying the condition (3) of Lemma \ref{lem:mod}.

\begin{lem}   
\label{lem:join}
For $1\leq k < n$ let $x\in \Sigma_n$ be any point whose stabilizer under the action of $\mathfrak S_n$ is the subgroup $G:=\mathfrak S_k\times \mathfrak S_{n-k}$.
Let $f_1,\dots,f_\ell:\Sigma_k*\Sigma_{n-k}\to\Sigma_k*\Sigma_{n-k}$ be $G$-equivariant maps with degrees $ d_1,\dots,  d_\ell$, respectively.

Then for any choice of the numbers $\varepsilon_i\in\{0,1\}$ there exists a $G$-equivariant map $\Sigma_n\to\Sigma_n$ which is the identity in a neighborhood of $x$ and whose degree is
\[
1 + \varepsilon_1( d_1-1) + \varepsilon_2( d_2- d_1) + \dots + \varepsilon_{\ell}( d_\ell- d_{\ell-1}) - \varepsilon_{\ell+1} d_\ell.
\]
\end{lem}

\begin{cor}
\label{cor:join}
Using the notation from the statement of Lemma \ref{lem:join}, suppose there exists a $G$-equivariant map $\Sigma_k*\Sigma_{n-k}\to\Sigma_k*\Sigma_{n-k}$ of degree $-1$. Then for any $ d$ there exists a $G$-equivariant map $\Sigma_n\to\Sigma_n$ which is the identity in a neighborhood of $x$ and whose degree is $ d$.
\end{cor}
\begin{proof}
The identity map $\Sigma_k*\Sigma_{n-k}\to\Sigma_k*\Sigma_{n-k}$ has degree $1$ and is $G$-equivariant. So, we can use $\pm 1$ for $ d_i$ in the statement of Lemma \ref{lem:join}.

Suppose we were able to achieve some degree using some values for $ d_1,\dots, d_\ell$ and $\varepsilon_1,\dots,\varepsilon_\ell$. It's sufficient to prove that we can change this number by $1$ and by $-1$ by incrementing $\ell$ and making a correct choice of $ d_{\ell+1}$ and $\varepsilon_{\ell+2}$.

When we increase $\ell$ by $1$ the degree changes by $w:=\varepsilon_{\ell+1} d_{\ell+1}-\varepsilon_{\ell+2} d_{\ell+1}=(\varepsilon_{\ell+1}-\varepsilon_{\ell+2}) d_{\ell+1}$. For any value of $\varepsilon_{\ell+1}\in{\{0,1\}}$, we can choose $\varepsilon_{\ell+2}$ so that $|\varepsilon_{\ell+1}-\varepsilon_{\ell+2}|=1$. Then choosing $ d_{\ell+1}$ to be either $1$ or $-1$, we can get $w=1$ and $w=-1$.
\end{proof}

\begin{cor}
\label{cor:join1}
Using the notation from the statement of Lemma \ref{lem:join}, suppose there exists a $G$-equivariant map $\Sigma_k*\Sigma_{n-k}\to\Sigma_k*\Sigma_{n-k}$ of degree $ d$. Then there exists a $G$-equivariant map $\Sigma_n\to\Sigma_n$ which is the identity in a neighborhood of $x$ and whose degree is also $ d$.
\end{cor}
\begin{proof}
In the statement of Lemma \ref{lem:join}, put $\ell=1$, $ d_1=d$, $\varepsilon_1 =1$, $\varepsilon_2 =0$. The corollary follows.
\end{proof}

\begin{proof}[Proof of Lemma \ref{lem:join}]
Using the radial projection we can $G$-equivariantly identify $\Sigma_n$ with the unit sphere $S$. Draw the diameter containing $x$. Draw $\ell+1$ different hyperplanes orthogonal to the diameter and intersecting its interior. The hyperplanes cut $S$ into two spherical caps $U_1$ and $U_2$ which are $G$-equivariantly homeomorphic to a cone over $\Sigma_k*\Sigma_{n-k}$, where $U_1$ contains $x$ and $U_2$ contains the point opposite to $x$; and $\ell$ cylinders $C_i$, each $G$-equivariantly homeomorphic to $\Sigma_k*\Sigma_{n-k}\times [0,1]$. We identify each $C_i$ with $\Sigma_k*\Sigma_{n-k}\times [0,1]$ using this  $G$-equivariant homeomorphism. For each $i$, the end $\Sigma_k*\Sigma_{n-k}\times \{1\}$ of $C_i$ is the end which is further away from $x$.

Now we construct a required map $f:\Sigma_n\to\Sigma_n$ defining it on the caps $U_1$, $U_2$, and on the cylinders $C_1,\dots,C_\ell$.
By construction, $f(\partial U_1)$, $f(\partial U_2)$, and $f(\partial C_1),\dots,f(\partial C_\ell)$ will be disjoint with $x$. So,
the degree of $f$ over $x$ will be the sum of the degrees of $f|_{U_1}$, $f|_{U_2}$, and $f|_{C_1},\dots,f|_{C_\ell}$ over $x$.

Define $f|_{U_1}$ to be the identity. Clearly, the degree of $f|_{U_1}$ over $x$ is $1$.

For each $i$, let $f$ map the end $\Sigma_k*\Sigma_{n-k}\times \{1\}$ of the cylinder $C_i$ to itself by the map $f_i$.
So, $f(\partial C_i)\subset \partial C_i$ for all $i$.

Let us now define $f$ on the interior of $C_i$.
Suppose that $\varepsilon_i = 1$. The restrictions of $f$ to the two boundary components of $C_i$ are $G$-equivariantly homotopic to each other as maps to $U_1\cup C_1\cup\dots\cup C_i$ because the latter space is $G$-equivariantly contractible. So, we can extend $f$ to the interior of $C_i$ using any such homotopy. The degree of $f|_{C_i}$ over $x$ is then the difference $d_i- d_{i-1}$ between the degrees of its restrictions to the boundary components considered as maps $\Sigma_k*\Sigma_{n-k}\to \Sigma_k*\Sigma_{n-k}$.

Suppose now that $\varepsilon_i = 0$. Again, the restrictions of $f$ to the two boundary components of $C_i$ are $G$-equivariantly homotopic to each other as maps to $C_{i}\cup\dots\cup C_{\ell}\cup U_2$. So, we can extend $f$ to the interior of $C_i$ using any such homotopy. The degree of $f|_{C_i}$ over $x$ is then $0$, since the map misses $x$.

In total, the degree of $f|_{C_i}$ over $x$ is $\varepsilon_i(d_i- d_{i-1})$.

The restriction $f|_{\partial U_2}$ is already defined to be $f_\ell$. Extend $f$ to $U_2$  by
some $G$-equivariant map going to either $U_2$ or its complement $S\setminus
U_2$ according to the value $\varepsilon_{\ell+1}=0$ or $\varepsilon_{\ell+1}=1$, respectively. The extension is possible since both $U_2$ and $S\setminus
U_2$ are $G$-equivariantly contractible. If $\varepsilon_{\ell+1}=0$, the degree of $f|_{U_2}$ over $x$ is $0$ since the map misses $x$. Otherwise, the degree equals $-d_\ell$, i.e., minus the degree of the restriction to the boundary considered as a map $\Sigma_k*\Sigma_{n-k}\to \Sigma_k*\Sigma_{n-k}$. In total, the degree of $f|_{U_2}$ over $x$ is $-\varepsilon_{\ell+1}d_\ell$.

So, the total degree of $f$ is as required.
\end{proof}

The last lemma we need is used only in the proof of part (d) of Theorem \ref{thm:main}.

\begin{lem}
\label{lem:representation_of_zero}
Let $n$ be a positive integer which is not a prime power and not twice a prime power. Then there exist integers $d_1, d_2, \dots, d_{n-1}$ such that 
\begin{itemize}
\item $1-\sum \limits_{k = 1}^{n-1}d_k\binom{n}{k}=0$,
\item $d_{q^\alpha}=0$ or $d_{q^\alpha} \equiv 1 \pmod{q}$ for any prime $q$ and $\alpha > 0$,
\item $d_1 \equiv 1 \pmod{p}$ if $n=p^t+1$ for some prime $p$.
\end{itemize}
\end{lem}

\begin{proof}
Consider all distinct representations of $n$ as a sum of two powers of the same prime, $n=p_1^{s_1}+p_1^{t_1}=p_2^{s_2}+p_2^{t_2}=\dots =p_\ell^{s_\ell}+p_\ell^{t_\ell}$, $0\leq s_i<t_i$ for each $i=1,\dots, \ell$. Note that it is possible that $s_i=0$ for some $i$. Clearly, $p_i\ne p_j$ for all $1\le i<j\le \ell$.

Set 

\begin{itemize}
\item $d_k = 0$ if $k = p_i^{t_i}$ for some $i$,
\item $d_k=1 + p_ib_k$ if $k = p_i^{s_i}$ for some $i$,
\item $d_k = 1+qb_k$ if $k\neq p_i^{s_i}$ and $k\neq p_i^{t_i}$ for all $i$ and $k=q^{\alpha}$, $\alpha>0$ for some prime $q$,
\item $d_k=b_k$ otherwise,
\end{itemize}
where integers $b_k$ will be chosen later. It is easy to see that the last two conditions on $d_k$ in the statement of the lemma are satisfied by this assignment.

Define the number $$N = 1 - \sum \limits_{k=p_i^{s_i}}\binom{n}{k} - \sum \limits_{k\neq p_i^{s_i},\;k\neq p_i^{t_i},\;k=q^{\alpha},\;\alpha>0} \binom{n}{k}.$$ Here the summation is over $k=1,\dots,n-1$ satisfying the second or the third case above. 
Define numbers $c_k$ as follows:
\begin{itemize}
\item $c_k = 0$ if $k = p_i^{t_i}$ for some $i$,
\item $c_k = p_i\binom{n}{k}$ if $k = p_i^{s_i}$ for some $i$,
\item $c_k = q\binom{n}{k}$ if $k\neq p_i^{s_i}$ and $k\neq p_i^{t_i}$ for all $i$ and $k=q^{\alpha}$, $\alpha>0$ for some prime $q$,
\item $c_k = \binom{n}{k}$ otherwise.
\end{itemize}

Plugging in these definitions we get
\[
1-\sum \limits_{k = 1}^{n-1}d_k\binom{n}{k}= N-\sum \limits_{k=1}^{n-1}b_kc_k.
\]
It remains to prove that we can choose $b_k$ so that the right-hand expression becomes $0$. This will follow if we prove that ${\rm GCD}(c_1,\dots,c_{n-1})$ divides $N$. To do that we first prove that $N$ is divisible by $p_1p_2\dots p_\ell$ and then prove that ${\rm GCD}(c_1,\dots,c_{n-1})$ divides $p_1p_2\dots p_\ell$.

By Lucas's theorem, for every $p_i$ and every $1\leq k\leq n-1$ the binomial coefficient $\binom{n}{k}$ is divisible by $p_i$, unless $k=p_i^{s_i}$ or $k=p_i^{t_i}$, in which case $\binom{n}{k}$ is equal $1$ modulo $p_i$. In the definition of $N$ above, for each $i$ there is a single summand $\binom{n}{p_i^{s_i}}$ and no summands $\binom{n}{p_i^{t_i}}$. Hence, $N$ is divisible by $p_i$ for every $i$.

Let us prove that ${\rm GCD}(c_1,\dots,c_{n-1})$ divides $p_1p_2\dots p_\ell$. Fix $i$. If $k=p^{s_i}$, then $c_k = p_i\binom{n}{p^{s_i}}$ is not divisible by $p_i^2$ because $\binom{n}{p^{s_i}}$ is not divisible by $p_i$. Hence, ${\rm GCD}(c_1,\dots,c_{n-1})$ is not divisible by $p_i^2$ for every $i$.

It remains to prove that ${\rm GCD}(c_1,\dots,c_{n-1})$ is not divisible by any prime $q$ which is not equal to any of $p_i$. To do so we find $c_k$ which is not divisible by $q$.

Suppose that $q>n$. Then $\binom{n}{k}$ is not divisible by $q$ for all $k$ and hence all the non-zero $c_k$ are also not divisible by $q$.

Suppose now that $q<n$. Write the base $q$ expansion of $n$ and decrease the leftmost digit by $1$, denoting the obtained number by $k$. Since $n>q$, the expansion had more than one digit and so $n-k$ is divisible by $q$. On the other hand, $n-p_i^{t_i}=p_i^{s_i}$ is not divisible by $q$, meaning that $k\neq p_i^{t_i}$ for all $i$. 

Also, $k$ is not a positive power of $q$. Indeed, assume the contrary. Then, by the definition of $k$, either $n=2k$, which is impossible because $n$ is not twice a prime power; or $n$ is the sum of $k$ and a larger positive power of $q$, which is impossible because $q$ is different from all $p_i$.

So, $k\neq p_i^{t_i}$ for all $i$ and $k$ is not a positive power of $q$. Hence, either $c_k=\binom{n}{k}$ (fourth case in the definition of $c_k$) or $c_k=q'\binom{n}{k}$ for some prime $q'\neq q$ (third case in the definition of $c_k$). Both numbers are not divisible by $q$ by Lucas's theorem because, by the definition of $k$, all the digits in the base $q$ expansion of $k$ are not greater than the corresponding digits of $n$. We have established that ${\rm GCD}(c_1,\dots,c_{n-1})$ divides $p_1p_2\dots p_\ell$.
\end{proof}

We are now ready to prove the rest of Theorem \ref{thm:main}.
\begin{proof}[\it Proof of Theorem \ref{thm:main}(a,c,d)]
For every $k=1,\dots,n-1$ pick the center $c_k$ of some $(k-1)$-dimensional face of $\Sigma_n$. The orbit of $c_k$ contains $\binom{n}{k}$ points and the stabilizer of $c_k$ in the permutation group is the subgroup $G_k := \mathfrak S_{k}\times \mathfrak S_{n-k}\subset\mathfrak S_n$. Denote $S_k:=\Sigma_k*\Sigma_{n-k}$. 

{\it Parts (a) and (c).} In (a) and (c) we have that $n$ is odd.
Since $n$ is odd, then one of the numbers $k$ and $n-k$ is even. The join of the antipodal map of the even dimensional factor and the identity map of the odd dimensional factor gives a $G_k$-equivariant map $S_k\to S_k$ of degree $-1$. By Corollary \ref{cor:join}, for any integer $d_k$ there exists a $G_k$-equivariant map $\Sigma_n\to \Sigma_n$ which is the identity in a neighborhood of $c_k$ and whose degree is $d_k$. By Lemma \ref{lem:mod}, there exists a $\mathfrak S_n$-equivariant map $\Sigma_n \to \Sigma_n$ of degree $$d=1 - \sum \limits_{k=1}^{n-1} d_k \frac{|\mathfrak S_n|}{|\mathfrak S_{k} \times \mathfrak S_{n-k}|}=1-\sum \limits_{k=1}^{n-1} d_k \binom{n}{k}.$$

If $n$ is not a prime power, by Lucas's theorem the GCD of the binomial coefficients in question is $1$. So, after an appropriate choice of $d_k$, the resulting degree $d$ can attain any integer value. This proves part (c) of the theorem.

Likewise, if $n$ is a power of a prime $p$, by Lucas's theorem the binomial coefficients in question are divisible by $p$.
On the other hand, $\binom{n}{n/p}$ is not divisible by $p^2$ by Kummer's theorem, and $\binom{n}{1}$ is not divisible by any prime except for $p$.
So, GCD of the coefficients equals exactly $p$ and after an appropriate choice of $d_k$, the resulting degree $d$ can attain any integer value which is $1$ modulo $p$. This finishes the proof of part (a) of the theorem, the ``only'' part of (a) was proved earlier.

{\it Part (d).} Let $d_k$ be some numbers whose existence is guaranteed by Lemma \ref{lem:representation_of_zero}. By Lemma \ref{lem:mod}, it is sufficient to prove that for each $k$ such that $d_k\neq 0$ there is a $G_k$-equivariant map $f_k:\Sigma_n\to \Sigma_n$ of degree $d_k$ which is the identity in a neighborhood of $c_k$. By Corollaries \ref{cor:join} and \ref{cor:join1}, this means that it is sufficient to find a $G_k$-equivariant map $S_k\to S_k$ of degree $-1$ or $d_k$. 

Finally, it is sufficient to find a $\mathfrak S_{k}$-equivariant map $\Sigma_k\to \Sigma_k$ or a $\mathfrak S_{n-k}$-equivariant map $\Sigma_{n-k}\to \Sigma_{n-k}$ of degree $-1$ or $d_k$. Indeed, using the the join operation with the identity map $\Sigma_{n-k}\to \Sigma_{n-k}$ or $\Sigma_k\to \Sigma_k$, respectively, we can get a required map $S_k\to S_k$.

Consider now all the possibilities for $k$.

$k$ is even: The antipodal map $\Sigma_k\to \Sigma_k$ is $\mathfrak S_{k}$-equivariant and of degree $-1$.

$k>1$ is odd and not a prime power: Then there is a $\mathfrak S_{k}$-equivariant map $\Sigma_k\to \Sigma_k$ of any degree, including $-1$, by part (c) of the theorem.

$k>1$ is odd and is a prime power with the base $p$: Then by the definition either $d_k=0$ and there is nothing to prove; or $d_k\equiv 1 \pmod{p}$. In the latter case, by part (a) of the theorem, there is a $\mathfrak S_{k}$-equivariant map $\Sigma_k\to \Sigma_k$ of degree $d_k$.

$k=1$ and $n=p^t+1$ for some prime $p$: Then $d_k\equiv 1 \pmod{p}$ by the definition. So, by part (a) of the theorem, there is a $\mathfrak S_{n-k}$-equivariant map $\Sigma_{n-k}\to \Sigma_{n-k}$ of degree $d_k$.

$k=1$ and $n\neq p^t+1$ for any prime $p$: Then $n-k$ is odd and not a prime power. So, by part (c) of the theorem, there is a $\mathfrak S_{n-k}$-equivariant map $\Sigma_{n-k}\to \Sigma_{n-k}$ of any degree including $-1$.
\end{proof}

%%%%%%%%%%%%%%%%%%%%%%%%%%%%%%%%%%%%%%%%%%%%%%%%%%%%%%%%%%%%%%%%%%%%%%%%%%%%%%%%%%%%%%%%%%%%%%%%%%%%%%%%%%%%%%%%%
%%%%%%%%%%%%%%%%%%%%%%%%%%%%%%%%%%%%%%%%%%%%%%%%%%%%%%%%%%%%%%%%%%%%%%%%%%%%%%%%%%%%%%%%%%%%%%%%%%%%%%%%%%%%%%%%%
\section{Appendix: Proof of Theorem \ref{thm:dold}}
Recall that $W_n$ denotes the affine span of $\Delta^{n-1}$. 
The group $\mathfrak S_n$ acts on $\Delta^{n-1}$ by permuting the barycentric coordinates and the action naturally extends to $W_n$.

The space $W_n$ is an $(n-1)$-dimensional
representation of the group $\mathfrak S_n$.
Its Euler class in the
cohomology $H^{n-1}(\mathfrak S_n; \pm \Z)$ is nonzero, this is easily
inferred from the direct calculation of the restriction of this Euler
class to the group $(\mathbb Z_p)^k$ contained in $\mathfrak S_n$.

Here $\pm \Z$ is the ring $\mathbb Z$ with the action of $\mathfrak S_n$
by the permutation sign. One may avoid using the coefficients
non-trivially acted upon by the group by considering the cohomology modulo
$p$ and, for $p>2$, passing to the subgroup $A_n\subset\mathfrak S_n$ of
even permutations. In this case the coefficients will be the finite
field $\mathbb Z_p$ and the Euler class in question remains nonzero. In
view of this, in what follows we put $G = \mathfrak S_n$ for $p=2$,
$n=p^k$, and $G=A_n$ for $p>2$, $n=p^k$.

Let $S(W_n)$ be the sphere of this representation. The spectral sequence
  starting with the page
\[
E_2^{x,y} = H^x(G; H^y(S(W_n); \mathbb Z_p))
\]
has $E_\infty^{x,y}$ the graded module associated to $H_G^*(S(W_n);
\mathbb Z_p)$. The only non-trivial differential of this spectral
sequence is generated by $d_{n-1}$ that non-trivially sends the
generator of $E_{n-1}(0,n-2) = H^{n-2}(S(W_n);\mathbb Z_p)$ to the Euler
class of $W_n$ in $E_{n-1}(n-1,0)=H^{n-1}(G;\mathbb Z_p)$.

The similar spectral sequence for $X$ starts with the page
\[
E_2^{x,y} = H^x(G; H^y(X; \mathbb Z_p)),
\]
and has $E_\infty^{x,y}$ the graded module associated to $H_G^*(X;
\mathbb Z_p)$. From the connectivity assumption (the triviality of the
reduced modulo $p$ cohomology of $X$ up to dimension $n-2$ would also suffice)
it follows that all differentials $d_r$ up to $r\le n-2$ of this
spectral sequences cannot map anything non-trivially to the bottom row. Hence the bottom row
$E_\infty^{x, 0}$ remains intact in dimensions $x\le n-1$.

If an equivariant map $X\to \mathbb R^n\setminus D_n$ existed, this
would imply an equivariant map $X\to S(W_n)$ by the projection and
normalization. Then the functoriality of the two spectral sequences
gives a contradiction with the established behavior of the differentials
$d_{n-1} : E_{n-1}^{0,n-2}\to E_{n-1}^{n-1,0}$ of the two spectral
sequences (one is non-trivial while the other is zero, the functorial map between the codomains
of the differentials is the identity).\qed

\bibliography{source}
\bibliographystyle{abbrv}

\end{document}